\documentclass[12pt]{article}

\usepackage{graphics,amsmath,amssymb,amsthm,mathrsfs}

\newcommand{\rn}[1]{\mathbb{R}^{#1}}

\newcommand{\pdydx}[2]{\frac{\partial #1}{\partial #2}}

\newtheorem{thm}{Theorem}[section]
\newtheorem{lemma}[thm]{Lemma}

\newtheorem{prop}[thm]{Proposition}
\theoremstyle{definition}

\numberwithin{equation}{section}

\begin{document}

\bibliographystyle{amsplain}

\title{A Bilinear Estimate for Biharmonic Functions \\ in Lipschitz Domains}

\author{Joel Kilty\footnote{Department of Mathematics,
 Centre College, Danville, KY 40422.
\emph{Email}: \texttt{joel.kilty@centre.edu}}
\and Zhongwei Shen\footnote{Corresponding author, Department of Mathematics,
University of Kentucky, Lexington, Kentucky 40506.
\emph{Telephone}: 1-859-257-3470.
 \emph{Email}: \texttt{zshen2@email.uky.edu}}}

\date{}

\maketitle

\begin{abstract}
\noindent
We show that a bilinear estimate for biharmonic functions in a Lipschitz domain $\Omega$
is equivalent to the solvability of the Dirichlet problem for the biharmonic equation
in $\Omega$. As a result, we prove that
for any given bounded Lipschitz domain $\Omega$ in $\rn{d}$ and $1<q<\infty$,
the solvability of the $L^{q}$ Dirichlet problem for $\Delta^2 u=0$ in $\Omega$
with boundary data in
$\mbox{\emph{WA}}^{1,q}(\partial\Omega)$
is equivalent to that of the $L^p$ regularity problem for $\Delta^2 u=0$ in $\Omega$
with boundary data in
$\mbox{\emph{WA}}^{2,p}(\partial\Omega)$,
where
$\frac{1}{p} +\frac{1}{q}=1$.
This duality relation, together with known results on the Dirichlet problem,
allows us to solve the $L^p$ regularity problem
for $d\ge 4$ and $p$ in certain ranges.


\end{abstract}

\section{Introduction}

Let $\Omega$ be a bounded Lipschitz domain in $\rn{d}$, $d\ge 2$.
We will use
$\mbox{\emph{WA}}^{k,p}(\partial\Omega)$
to denote the completion of
the set of arrays of functions
\begin{equation} \label{defnWhitneyArray}
\big\{\dot{f}=(f_{\alpha})_{|\alpha|\leq k}=(D^{\alpha} f|_{\partial\Omega} )_{|\alpha|\leq k}: \
f\in C^{\infty}_0(\rn{d}) \big\} \end{equation}
under the scale-invariant norm on
$\partial\Omega$,
\begin{equation} \label{whitneyArrayNorm}
\|\dot{f}\|_{k,p}
:= \sum_{|\alpha|\leq k} |\partial\Omega|^{\frac{k-|\alpha|}{1-d}} \|f_\alpha \|_{p},
\end{equation}
where
$\|\cdot \|_{p}$
denotes the norm in $L^p(\partial\Omega)$.
In this paper we are interested
in the Dirichlet problem for the biharmonic equation $\Delta^2 u=0$ in $\Omega$
with boundary data taken from $\mbox{\emph{WA}}^{k,p}(\partial\Omega)$ where $k=1$ or $2$.
Recall that the $L^p$ Dirichlet problem, denoted by $(D)_p$,
is said to be uniquely solvable
if given any
$\dot{f} \in \mbox{\emph{WA}}^{1,p}(\partial\Omega)$,
there exists a unique function $u$
such that
\begin{equation} \label{dirichletProb}
\left\{ \begin{array}{ll} \Delta^2 u=0 &
\mbox{ in } \Omega,\\
D^{\alpha} u = f_{\alpha} &
\mbox{ a.e. on } \partial\Omega \ \text{ for }\ |\alpha|\leq 1, \\
(\nabla u)^* \in L^p(\partial\Omega). & \end{array} \right.
\end{equation}
Moreover, the solution $u$ satisfies
the estimate \begin{equation} \label{Dirichlet-estimate}
\| (\nabla u)^*\|_p +|\partial\Omega|^{\frac{1}{1-d}}\| (u)^*\|_p
\leq C\sum_{|\alpha| \leq 1}  |\partial\Omega|^{\frac{1-|\alpha|}{1-d}}
\|f_{\alpha} \|_{p}.
\end{equation}
Here $(w)^*$ denotes the nontangential maximal function of $w$.

If the Dirichlet data in (\ref{dirichletProb}) are taken from
$\mbox{\emph{WA}}^{2,p}(\partial\Omega)$
instead of $\mbox{\emph{WA}}^{1,p}(\partial\Omega)$,
we may expect the solution to have one higher order regularity.
This is the so-called $L^p$ regularity problem.
Let $\nabla_t g$ denote the tangential derivatives of $g$ on $\partial\Omega$.
We say that the $L^p$ regularity problem for $\Delta^2 u=0$ in $\Omega$,
denoted by $(R)_p$,
is uniquely solvable if given
$\dot{g}=\{g_{\alpha}: |\alpha|\leq 2 \} \in \mbox{\emph{WA}}^{2,p}(\partial\Omega)$,
there exists a unique function $u$ such that
\begin{equation} \label{regularityProb}
\left\{ \begin{array}{ll} \Delta^2 u=0 & \mbox{ in } \ \Omega, \\
D^{\alpha} u = g_{\alpha} & \mbox{ a.e. on } \partial\Omega \ \text{ for }\  |\alpha|\leq 1, \\
(\nabla^2 u)^* \in L^p(\partial\Omega). & \end{array} \right.
\end{equation}
Moreover, the solution satisfies the estimate
\begin{equation} \label{regularity-estimate}
\aligned
\|(\nabla^2 u)^*\|_p
& +|\partial\Omega|^{\frac{1}{1-d}} \|(\nabla u)^*\|_p
+|\partial\Omega|^{\frac{2}{1-d}}\| (u)^*\|_p\\
&\leq C \bigg\{ \sum_{|\alpha| =1}
\|\nabla_t g_{\alpha} \|_p+
\sum_{|\alpha|\leq 1} |\partial\Omega|^{\frac{2-|\alpha|}{1-d}}\|g_{\alpha} \|_p
 \bigg\}
\endaligned
\end{equation}
and $\nabla^2 u$ has nontangential limits a.e. on $\partial\Omega$.

If $\Omega$ is a bounded $C^1$ domain, the $L^p$ Dirichlet problem $(D)_p$ for
the biharmonic equation is uniquely solvable for any $1<p<\infty$ \cite{cohen-gosselin, v:c-1}.
By the duality argument in \cite{verchota:polyharmonic} for star-shaped Lipschitz domains
(or by Theorem 1.1 below),
this implies that the $L^p$ regularity problem $(R)_p$ for the biharmonic equation
in $C^1$ domains is uniquely solvable for $1<p<\infty$.
However, it has been known that given any $p<2$ or $p$ large in the case $d\ge 4$,
there is a bounded Lipschitz domain for which
 the $L^p$ Dirichlet problem is not solvable
 (see e.g. \cite{dkv:biharmonic,pv:biharmonic}).

The main result in this paper is the following duality relation
between the Dirichlet and regularity
problems for the biharmonic equation in Lipschitz domains.

\begin{thm} \label{Main-Theorem}
     Let $1<p<\infty$ and $\frac{1}{p}+\frac{1}{q}=1$.
Let $\Omega$ be a bounded Lipschitz domain in $\rn{d}$, $d\ge 2$.
For the biharmonic equation $\Delta^2 u=0$,
 the Dirichlet problem $(D)_q$ in $\Omega$ is uniquely solvable
if and only if the regularity problem $(R)_p$ in $\Omega$ is
uniquely solvable.
\end{thm}

Theorem \ref{Main-Theorem} extends a similar result in \cite{kilty:higherOrderRegularity}
 for second order elliptic systems.
We mention that although it is not explicitly stated, the duality relation between
the Dirichlet and regularity problems was essentially
established by G. Verchota in the case of star-shaped Lipschitz domains
for Laplace's equation in \cite{verchota}.
Verchota's duality argument relies on
the representation of solutions by layer potentials in terms of their
Cauchy data,
and exploits the square function estimates and
a well known fact that there exists
a local pointwise estimate between a square function of $D_i u$ and
a square function of $D_d u$ \cite{stein:singularintegrals}.
A similar idea was used in \cite{verchota:polyharmonic}
to establish the solvability of the $L^2$ regularity problem
for the polyharmonic equation in Lipschitz domains.
The approach in \cite{kilty:higherOrderRegularity} and in this paper
uses the basic duality argument of Verchota.
One of our main contributions is a localization argument which allows us to
establish the duality relation
between the Dirichlet and regularity problems for each individual
Lipschitz domain, not just the class of all Lipschitz domains in $\rn{d}$.

More precisely, to prove Theorem \ref{Main-Theorem}, we introduce a bilinear form
\begin{equation} \label{bilinear-form}
\Lambda[f,g] := \lim_{m\rightarrow\infty} \int_{\partial\Omega_m}
\left\{ \pdydx{}{n}\Delta u\cdot v - \Delta u\cdot \pdydx{v}{n} \right\}\,d\sigma
=-\int_\Omega \Delta u\cdot \Delta v\, dx,
\end{equation}
for $f, g\in C_0^2 (\rn{d})$,
where $\Omega_m\uparrow \Omega$ and
$u,v$ are solutions of the $L^2$ regularity problem for the biharmonic
equation in $\Omega$ with boundary data $D^\alpha f,D^\alpha g$, respectively.
The key new idea in this paper
 is to show that both the solvability of $(D)_q$ and the solvability of $(R)_p$
are equivalent to the following bilinear estimate,
\begin{equation}\label{bilinear-estimate}
\big|\Lambda[f,g]\big|\le
C \bigg\{
\|\nabla_t \nabla f\|_p +|\partial\Omega|^{\frac{1}{{1-d}}} \| \nabla f\|_p
+|\partial\Omega|^{\frac{2}{1-d}}\| f\|_p\bigg\}
\bigg\{ \|\nabla g\|_q +|\partial\Omega|^{\frac{1}{1-d}} \| g\|_q\bigg\}.
\end{equation}

\begin{thm}\label{Main-Theorem-1}
Let $1<p<\infty$ and $\frac{1}{p} +\frac{1}{q}=1$.
For $\Delta^2 u=0$ in a bounded Lipschitz domain
$\Omega$ in $\rn{d}$, $d\ge 2$,  the following statements are equivalent.

\item{1.} The $L^q$ Dirichlet problem $(D)_q$ is uniquely solvable in $\Omega$.

\item{2.} The $L^p$ regularity problem $(R)_p$ is uniquely solvable in $\Omega$.

\item{3.} The bilinear estimate (\ref{bilinear-estimate}) holds for any $f, g\in C_0^\infty(\rn{d})$.

\end{thm}

There exists an extensive literature
on $L^p$ boundary values problems in Lipschitz domains
for $p$ near 2 and $d\ge 2$ (see e.g. \cite{kenig} for references).
In particular the solvability of the $L^2$ Dirichlet and regularity problems for the
biharmonic equation in Lipschitz domains was established
in \cite{dkv:biharmonic,verchota:polyharmonic}.
We mention that in the lower dimensional case $d=2,3$, the sharp ranges of
$p$'s for which the $L^p$ boundary value problems are solvable are known
for elliptic systems and higher order elliptic equations
(see \cite{dahlberg4,pv:biharmonic,pv:maximum,verchota:higherOrder}).
Results on the inhomogeneous Dirichlet problem for $\Delta^2$ as well as
for general higher order elliptic systems
may be found in \cite{adolfsson-pipher, mitrea-mitrea-pipher-wright,
mazya-mitrea-shaposhnikova, mitrea-mitrea}.
The $L^p$ boundary value problems in higher dimensional case $d\geq 4$ were studied by Shen
in \cite{shen:ellip,shen:necsuff,shen:biharmonic,shen:boundary}.
Related work may be found
in \cite{kilty,shen:stokes,wright} for the Stokes system.
These results extend the classical work of Dahlberg, Jerison, Kenig, and Verchota
in \cite{dahlberg2,dahlberg1,dahlberg,jerison,jerison-kenig, verchota}
on $L^p$ boundary value problems for Laplace's equation in Lipschitz domains.

More specifically,
for the biharmonic equation in Lipschitz domains, the $L^p$ Dirichlet problem is solvable
for $2-\varepsilon<p\le \infty$ if $d=2$ or $3$; for $2-\varepsilon<p<6 +\varepsilon$ if
$d=4$; for $2-\varepsilon<p<4+\varepsilon$ if $d=5,6,7$; and for $2-\varepsilon
<p< p_d +\varepsilon$ if $d\ge 8$, where
\begin{equation}\label{range of p}
p_d=2+\frac{4}{d-\lambda_d} \text{ and }
\lambda_d =\frac{d+10 +2\sqrt{2(d^2-d +2)}}{7},
\end{equation}
(see \cite{dkv:biharmonic,pv:biharmonic,pv:maximum,shen:ellip, shen:necsuff,shen:biharmonic};
the ranges of $p$ are known to be sharp for $2\le d\le 7$).
When combined with these results on the Dirichlet problem,
Theorem 1.1 allows us to solve the $L^p$ regularity problem for the biharmonic
equation in Lipschitz domains.
 The results in the following theorem are new in the case $d\ge 4$ and $p<2$.

\begin{thm} \label{regularitySolvability}
    Let $\Omega$ be a bounded Lipschitz domain in $\rn{d}$.
 Then the $L^p$ regularity problem for $\Delta^2 u=0$ in $\Omega$
is uniquely solvable if
    \begin{eqnarray*}
1<p< 2+\varepsilon & & \mbox{ if } \ d=2, 3,\\
        (6/5)-\varepsilon < p < 2+\varepsilon && \mbox{ if } \ d=4, \\
        (4/3)-\varepsilon < p < 2+\varepsilon && \mbox{ if }  \ d=5,6,7, \\
        q_d - \varepsilon < p < 2+\varepsilon && \mbox{ if } \ d\geq 8,
    \end{eqnarray*}
where  $\varepsilon>0$ depends on $\Omega$, $q_d=p_d/(p_d-1)$
and $p_d$ is given by (\ref{range of p}).
Moreover, the ranges of $p$'s are sharp for $2\le d\le 7$.
\end{thm}

For a general homogeneous elliptic system of order $2\ell$
with constant coefficients,
it was proved in \cite{shen:necsuff} that given any $p>2$,
the solvability of $(D)_p$ is
equivalent to a weak reverse H\"older condition with exponent $p$.
A similar result was established in \cite{kilty:higherOrderRegularity}
for $(R)_p$ with $p>2$.
It follows from Theorem \ref{Main-Theorem} that for the biharmonic equation,
the solvability of $(D)_p$ for some $p<2$ is equivalent to a
weak reverse H\"older condition with exponent $q=\frac{p}{p-1}$.
Since the weak reverse H\"older conditions have the self-improving property,
we obtain the following.

\begin{thm}\label{self-improving}
For the biharmonic equation in a bounded Lipschitz domain $\Omega$,
the set of all exponents $p\in (1,\infty)$ for which $(D)_p$ in $\Omega$ is uniquely
solvable is open.
\end{thm}

Finally we note that any convex domain is a Lipschitz domain. It was proved in \cite{shen:biharmonic}
that the $L^p$ Dirichlet problem for the biharmonic equation in convex domains
is solvable for $2\le p<\infty$. Theorem \ref{Main-Theorem}, together with
recent results on the boundedness of $\nabla^2 u$ for biharmonic functions in convex domains
in \cite{mayboroda} and Theorem 1.1
in \cite{kilty:higherOrderRegularity}, allows us to solve the $L^p$ Dirichlet and regularity
problems on convex domains for any $1<p<\infty$.

\begin{thm}\label{convex-theorem}
For the biharmonic equation in a convex domain $\Omega$ in $\rn{d}$,
 the $L^p$ Dirichlet and regularity problems are uniquely solvable for any $1<p< \infty$.
Furthermore, the solution of the $L^p$ Dirichlet problem with boundary data
$D^\alpha u\in L^\infty(\partial\Omega)$ for $|\alpha|\le 1$ satisfies
the weak maximum principle
\begin{equation}\label{weak-max}
\|\nabla u\|_{L^\infty(\Omega)}
\le C\, \| \nabla u\|_{L^\infty(\partial\Omega)}.
\end{equation}
\end{thm}

\noindent{\bf Remarks on notations.}
 The summation convention will be used throughout this paper.
Also, $\Omega$ will denote a bounded Lipschitz domain in $\rn{d}$, $d\ge 2$ and
$n$  the unit outward normal to $\partial\Omega$.
By rescaling we may always assume that $|\partial\Omega|=1$.
We will use $(D)_p$ and $(R)_p$ to denote the $L^p$ Dirichlet and regularity
problems respectively for $\Delta^2 u=0$ in $\Omega$.

\section{Preliminary estimates}

In this section we collect several estimates
that will be needed later in the paper.
We start with definitions of the nontangential maximal function and the square function.
Given a regular family of truncated nontangential cones
$\{ \gamma(Q): Q\in \partial\Omega\}$ for $\Omega$,
the nontangential maximal function $(w)^*$ is defined by
\begin{equation}
(w)^*(Q)=\sup\big\{ |w(x)|:\ x\in\gamma(Q)\big\}
\end{equation}
and the square function $S(w)$ by
\begin{equation}
S(w)(Q)=\left\{ \int_{\gamma(Q)} \frac{|\nabla w(x)|^2}{|x-Q|^{d-2}}\, dx\right\}^{1/2}.
\end{equation}
It was proved in \cite{pv:area-integral} that if $\Delta^2 u=0$ in $\Omega$, then
\begin{equation}\label{maximal-dominate}
\| S(\nabla u)\|_p \le C\, \|(\nabla u)^*\|_p
\end{equation}
and
\begin{equation}\label{square-dominate}
\| (\nabla u)^*\|_p \le C\, \bigg\{ \| S(\nabla u)\|_p
+\sup_K |\nabla u|\bigg\},
\end{equation}
where $0<p<\infty$ and $K\subset \Omega$ is compact.

\begin{lemma} \label{squareFunctionProp}
    Let $P\in \partial\Omega$.
Suppose that for some constants $C_0>C_1>10$,
we have $\Omega\subset B(P,C_0r)$
and
\begin{equation}\label{coordinate-represntation}
B(P,C_1r)\cap \Omega = B(P,C_1r) \cap \big\{(x^\prime,x_d): x_d>\eta(x^\prime)\big\},
\end{equation}
where $\eta:\rn{d-1}\rightarrow \rn{}$ is Lipschitz.
Then there exist two regular families of truncated nontangential
cones $\{ \gamma (Q): Q\in \partial\Omega\}$,
$\{ \widetilde{\gamma}(Q): Q\in \partial\Omega\}$
such that for any biharmonic function $u$ in $\Omega$,
\begin{equation} \label{squareFunctionLastVariable}
S(\nabla u) (Q) \leq C \bigg\{
\widetilde{S}\left(D_d u\right) (Q) +\sup_K |\nabla u|\bigg\}
\quad \text{ for any } \ Q\in B(P, 5r)\cap \partial\Omega,
\end{equation}
where $S(w)$ and $\widetilde{S}(w)$ denote
the square functions defined by using $\gamma(Q)$ and $\widetilde{\gamma}(Q)$ respectively,
and $K\subset \Omega$ is compact.
\end{lemma}

\begin{proof}
Estimate (\ref{squareFunctionLastVariable}) follows from
the well known interior estimates for biharmonic functions
by an argument found in pp. 214-216 of
\cite{stein:singularintegrals}. We omit the details.
\end{proof}

\begin{lemma} \label{radialMaxLemma}
 Under the same assumptions as in Lemma \ref{squareFunctionProp}, we have
 \begin{equation} \label{radialMaxEst}
\int_{B(P,5r)\cap\partial\Omega} |\mathcal{M}(w)|^{p}\,d\sigma
\leq {C}\left\{ \int_{\Omega} |w|^{p}\,dx
+ \int_{\Omega} |\nabla w|^{p}\,dx \right\},
\end{equation}
where $w\in C^1(\Omega)$, $1\le p<\infty$, and
$\mathcal{M}(w)(Q) = \sup_{0<t<cr} |w(Q+te_d)|$.
\end{lemma}

\begin{proof}
We may assume that $r=1$. Writing
$$
w(Q+t_1 e_d) - w(Q+te_d)
= \int_t^{t_1} D_d w(Q+se_d)\,ds
$$
for $0<t<t_1<c$
and using H\"{o}lder's inequality, we obtain
    \begin{eqnarray*}
        |w(Q+te_d)|^p &\leq& C\, \left\{ |w(Q+t_1e_d)|^p
+ \int_0^c |D_d w (Q+s e_d) |^{p}\,ds\right\}.
    \end{eqnarray*}
Next integrating both sides in $t_1$ gives
 $$
|w(Q+te_d)|^p \leq C\left\{ \int_0^c |w(Q+se_d)|^p\,ds +
\int_0^c |D_d w(Q+se_d)|^p\,ds\right\}.
$$
The desired estimate now follows by taking the supremum in $t$
and then integrating in $Q$ over $B(P,5r)\cap \partial\Omega$.
\end{proof}

Throughout the remainder of the paper we will often need 
to approximate a Lipschitz domain $\Omega$ with a sequence of
$C^{\infty}$ domains $\Omega_m$.  Here we briefly recall some of the most important properties of this
approximation scheme, which can be found in \cite{verchota}.

\begin{thm}
    Let $\Omega\subset \rn{d}$ be a bounded Lipschitz domain.  Then
        \begin{enumerate}
            \item There is a sequence of $C^{\infty}$ domains, 
$\Omega_j\subset \Omega$, and homeomorphisms, $\Lambda_j:\partial\Omega\rightarrow
            \partial\Omega_j$, such that 
$\sup_{Q\in\partial\Omega} |Q-\Lambda_j(Q)|\rightarrow 0$ as $j\rightarrow\infty$ and for all $j$ and all 
            $Q\in\partial\Omega$, $\Lambda_j(Q)\in \gamma(Q)$.  Here $\{ \gamma(Q): Q\in \partial\Omega\}$ 
is a family of regular nontangential cones associated with $\Omega$.
            
            \item There are positive functions $\omega_j:\partial\Omega\rightarrow \rn{}$ bounded away from zero and infinity uniformly in $j$ such that
            for any measurable set $E\subset \partial\Omega$, $\int_E \omega_j\,d\sigma = \int_{\Lambda_j(E)}\,d\sigma_j$, and so that $\omega_j\rightarrow 1$
            pointwise a.e. and in $L^q(\partial\Omega)$, $1\leq q < \infty$.
            
            \item The normal vectors to $\Omega_j$, $n(\Lambda_j(Q))$, converge pointwise a.e.  and in every $L^q(\partial\Omega)$, $1\leq q<\infty$ to
            $n(Q)$.
            
            \item There exists a $C^{\infty}$ vector field 
$\mathbf{h}$ in $\rn{d}$ such that for all $j$ and $Q\in \partial\Omega$ 
$$<\mathbf{h}(\Lambda_j(Q)),
            n(\Lambda_j(Q))>\geq c>0,
$$
 where $c$ depends only on $d$ and the Lipschitz character of $\Omega$.
        \end{enumerate}
\end{thm}

We will write $\Omega_j \uparrow \Omega$ to indicate such a sequence of approximating domains.
\begin{lemma} \label{sobolevLimit}
    Let $p\ge 2$ and $\frac{1}{p}+\frac{1}{q}=1$.
Suppose that $\Delta w=0$ in $\Omega$,
$w=g\in L^p(\partial\Omega)$ on $\partial\Omega$, and $(w)^*\in L^p(\partial\Omega)$.
Then $\ell=\pdydx{w}{n} \in W^{-1,p}(\partial\Omega)$ in the sense that
for any $f\in W^{1,q}(\partial\Omega)$,
\begin{equation}
\ell( f) = \lim_{m\to\infty} \int_{\partial\Omega_m} \pdydx{w}{n}\cdot F\,d\sigma
\end{equation}
where $\Omega_m\uparrow \Omega$ and
$F$ is any function with the property that
$F\in C^1(\Omega\setminus K)$ for some compact subset $K$ of $\Omega$,
$(\nabla F)^*\in L^{q}(\partial\Omega)$,
and $F=f$ a.e. on $\partial\Omega$ in the sense of nontangential convergence.
  Moreover,
\begin{equation}
\|\ell\|_{W^{-1,p}(\partial\Omega)}\le C\, \|g\|_p
\end{equation}
and
\begin{equation}
\ell(f) = \int_{\partial\Omega} w\cdot \pdydx{v}{n}\,d\sigma,
\end{equation}
 where
\begin{equation}\label{Laplace-regularity}
\left\{ \begin{array}{ll} \Delta v=0 & \mbox{ in } \ \Omega, \\
v=f & \mbox{ on } \ \partial\Omega, \\
(\nabla v)^* \in L^{q}(\partial\Omega).
\end{array}\right.
\end{equation}
\end{lemma}

\begin{proof} The proof uses the fact that the $L^q$ regularity problem (\ref{Laplace-regularity})
for Laplace's equation is solvable in any Lipschitz domain for $1<q\le 2$.
See Lemma 4.9 and its proof in \cite{verchota:biharmonic}.
\end{proof}

\begin{lemma} \label{nonTangLimitLemma}
    Suppose that $\Delta^2 u=F$ in $\Omega$ and $F\in C(\overline{\Omega})$. 
If the sequence $\{ \|\nabla^{3} u\|_{L^p(\partial\Omega_m)}\}$ is bounded for some $p>1$ and
$\Omega_m\uparrow \Omega$, then $\nabla^{3} u$ has nontangential limits
    a.e. on $\partial\Omega$.  Furthermore, $\nabla^{3} u \in L^p(\partial\Omega)$ and
    \begin{equation}\label{nontangentialestimate}
        \|(\nabla^{3} u)^*\|_p \leq C\big\{  \|\nabla ^{3} u\|_p +\| F\|_{L^\infty(\Omega)} \big\},
    \end{equation}
    where $C$ depends only on $d$, $p$ and the Lipschitz character of $\Omega$.
\end{lemma}

\begin{proof} Let $\Gamma (x)$ denote the fundamental solution for the operator
$\Delta^2$ in $\rn{d}$, with pole at the origin.
By substracting $w(x)=\int_\Omega \Gamma (x-y) F(y)\, dy$ from $u$, one may assume that $F=0$.
In this case the conclusion of the lemma is proved in \cite{kilty:higherOrderRegularity} (see Theorem 2.1)
under the assumption $(\nabla^3 u)^*\in L^p(\partial\Omega)$.
An inspection of the proof of Theorem 2.1 shows that
the same argument
goes through under the weaker assumption that the sequence
$\{ \|\nabla^{2\ell -1} u\|_{L^p(\partial\Omega_m)}\}$ is bounded, where $u$ is a solution
of a general homogeneous elliptic system of order $2\ell$ with constant coefficients.
We omit the details.
\end{proof}

We end this section with a Green's identity for the bi-Laplacian $\Delta^2$,
\begin{equation}\label{Green's-identity}
\aligned
\int_{\partial\Omega} \bigg\{
\frac{\partial}{\partial n} \Delta u \cdot v
- & \Delta u\cdot \frac{\partial v}{\partial n}\bigg\} d\sigma
-\int_{\partial\Omega}
\left\{
\frac{\partial}{\partial n} \Delta v \cdot u-\Delta v\cdot \frac{\partial u}{\partial n}\right\}
d\sigma\\
& =\int_{\Omega}
\big\{ v\Delta^2 u -u\Delta^2 v\big\} dx
\endaligned
\end{equation}
for any $u, v\in C^4(\overline{\Omega})$.

\section{Uniqueness}

In this section we prove two theorems on the uniqueness of the Dirichlet and regularity
problems for the biharmonic equation in Lipschitz domains.
Let $\Gamma ^x(y)=\Gamma (y-x)$, where $\Gamma(y)$ denotes the fundamental solution of
$\Delta^2 u=0$ in $\rn{d}$ with pole at the origin.

\begin{thm}\label{regularity-uniqueness}
Let $1<p<\infty$ and $\frac{1}{p}+\frac{1}{q}=1$.
Suppose that for each $x\in \Omega$, the $L^{q}$ Dirichlet
problem for the biharmonic equation
in $\Omega$ with boundary data $D^\alpha \Gamma^x$
has a solution $V^x$. Then the solution of the $L^{p}$
regularity problem for the biharmonic equation in $\Omega$
with any given boundary data in $\mbox{\emph{\it WA}}^{2,p}(\partial\Omega)$,
 if it exists,
is unique.
\end{thm}

\begin{proof}
Fix $x\in \Omega$ and let $G^x=\Gamma ^x-V^x$.
For $0<\varepsilon<(1/4)\text{dist}(x, \partial\Omega)$,
choose $\varphi =\varphi_\varepsilon$ so that
$\varphi=1$ in $\{ y\in \Omega: \text{dist}(y, \partial\Omega)\ge 2\varepsilon\}$,
$\varphi=0$ in $\{ y\in \Omega: \text{dist}(y,\partial\Omega)\le (1/2)\varepsilon\}$,
and $|D^\alpha \varphi|\le C\varepsilon^{-|\alpha|}$ for $|\alpha|\le 4$.
Suppose now that $\Delta^2 u=0$ in $\Omega$. We may write
$$
\aligned
u(x) & =\int_\Omega G^x \Delta^2 (u\varphi)\, dy\\
&=\int_\Omega
G^x \big\{ 4D_i\Delta u \cdot D_i \varphi +2\Delta u\cdot \Delta \varphi
+4 D_iD_j u\cdot D_iD_j \varphi
+4D_i u \cdot D_i \Delta \varphi
+u\Delta^2 \varphi\big\}\, dy\\
&=-4\int_\Omega D_i G^x \cdot\Delta u \cdot D_i\varphi \, dy\\
&\ \ \ \ \ \ \
+\int_\Omega
G^x \big\{ -2 \Delta u\cdot \Delta \varphi
+ 4 D_iD_j u\cdot D_iD_j \varphi
+4D_i u \cdot D_i \Delta \varphi
+u\Delta^2 \varphi\big\}\, dy.
\endaligned
$$
It follows that
\begin{equation}\label{Green-Estimate-1}
|u(x)|
\le C
\int_{K_\varepsilon}
\big\{ \varepsilon^{-1}|\nabla G^x|\, |\Delta u|
+\varepsilon^{-2} |G^x|\, |\nabla^2 u|
+\varepsilon^{-3} |G^x|\, |\nabla u|
+\varepsilon^{-4} |G^x|\, |u|\big\} \, dx,
\end{equation}
where $K_\varepsilon =\{ x\in \Omega: (1/2)\varepsilon\le
\text{dist}(x, \partial\Omega)\le 2 \varepsilon\}$.

Let
\begin{equation}\label{truncted-maximal}
(w)^*_\varepsilon (Q)
=\sup\big\{ |w(x)|:\ x\in \gamma(Q)
\text{ and } \text{dist}(x, \partial\Omega)\le 2\varepsilon\big\}.
\end{equation}
Since $G^x=0$ on $\partial\Omega$, we have $|G^x(y)|\le C \varepsilon
\, (\nabla G^x)^*_\varepsilon (Q)$ for
any $y\in \gamma (Q)$ with dist$(y, \partial\Omega)\le 2\varepsilon$.
Similarly, if $u=|\nabla u|=0$ on $\partial\Omega$, we have
$|u(y)|\le C\varepsilon^2 (\nabla^2 u)^*(Q)$ and $|\nabla u(y)|\le C\varepsilon (\nabla^2 u)^*(Q)$
for any $y\in \gamma(Q)$ with dist$(y, \partial\Omega)\le 2\varepsilon$.
This, together with (\ref{Green-Estimate-1}), shows that
\begin{equation}\label{Green-Estimate-2}
|u(x)|\le C\int_{\partial\Omega} (\nabla G^x)^*_\varepsilon (\nabla^2 u)^*\, d\sigma
\end{equation}
if $\Delta^2 u=0$ and $u=|\nabla u|=0$ on $\partial\Omega$.

Finally, if $u$ is a solution of the $L^p$ regularity problem with zero data, then
$$
(\nabla G^x)^*_\varepsilon (\nabla^2 u)^* \in L^1 (\partial\Omega),
$$
as
$(\nabla G^x)^*_\varepsilon \in L^{q}(\partial\Omega)$
and $(\nabla^2 u)^*\in L^p(\partial\Omega)$.
Furthermore, since $|\nabla G^x|=0$ on $\partial\Omega$, we see that
$(\nabla G^x)^*_\varepsilon \to 0$ as $\varepsilon\to 0$.
Consequently, by Lebesgue's dominated convergence theorem,
the integral in (\ref{Green-Estimate-2}) converges to $0$ as
$\varepsilon\to 0$. It follows that $u(x)=0$ for any $x\in \Omega$.
This gives the uniqueness of the $L^p$ regularity problem.
\end{proof}

We have a similar result on the uniqueness of the $L^p$ Dirichlet problem.

\begin{thm}\label{Dirichlet-uniqueness}
Let $1<p<\infty$ and $\frac{1}{p}+\frac{1}{q}=1$.
Suppose that for each $x\in \Omega$, the $L^{p}$ regularity problem
for the biharmonic equation
in $\Omega$ with boundary data $D^\alpha \Gamma^x$
has a solution $W^x$. Then the solution of the $L^{q}$
Dirichlet problem for the biharmonic equation in $\Omega$
with any given boundary data in $\mbox{\emph{\it WA}}^{1,q}(\partial\Omega)$, if it exists,
is unique.
\end{thm}

\begin{proof} The proof is similar to that of Theorem \ref{regularity-uniqueness}.
Fix $x\in \Omega$ and let $\widetilde{G}^x=\Gamma^x-W^x$.
It follows from integration by parts that
if $\Delta^2 u=0$ in $\Omega$,
$$
\aligned
u(x) &=\int_\Omega \widetilde{G}^x \Delta^2(u\varphi)\, dy\\
&=4\int_\Omega D_iD_j \widetilde{G}^x\cdot D_j u \cdot D_i \varphi \, dy
+2 \int_\Omega D_i\widetilde{G}^x \cdot D_i u \cdot \Delta\varphi \, dy
\\
&\ \ \ \ \ \quad\quad
+\int_\Omega\widetilde{G}^x \big\{
2 D_i u \cdot D_i \Delta \varphi
+ u\Delta^2 \varphi \big\}\, dy.
\endaligned
$$
As in the proof of Theorem \ref{regularity-uniqueness}, this implies that
\begin{equation}\label{Green-Estimate-3}
|u(x)|
\le C\int_{\partial\Omega}
(\nabla^2 G^x)^*_\varepsilon (\nabla u)^*_\varepsilon\, d\sigma,
\end{equation}
if $u$ is a solution of the $L^{q}$ Dirichlet problem with zero boundary data.
Since
$$
(\nabla^2 G^x)^*_\varepsilon (\nabla u)^*_\varepsilon
\le (\nabla^2 G^x)^*_{\varepsilon_0} (\nabla u)^*
\in L^1(\partial\Omega)
$$
for $0<\varepsilon<\varepsilon_0=(1/4)\text{dist}(x, \partial\Omega)$
and $(\nabla u)^*_\varepsilon \to 0$ as $\varepsilon\to 0$,
we again obtain $u(x)=0$ for any $x\in \Omega$ by Lebesgue's dominated
convergence theorem.
\end{proof}

\section{A bilinear estimate}

Let $f,g$ be two $C^2$ functions in some neighborhood of $\partial\Omega$
and $u,v$ be solutions of
the $L^2$ regularity problem for the biharmonic equation in $\Omega$
with boundary data $f,g$ respectively, given in \cite{verchota:polyharmonic,pv:higherOrder};
that is,

$$
 \left\{ \begin{array}{ll} \Delta^2 u=0 & \mbox{ in } \Omega, \\
D^{\alpha}u=D^{\alpha} f & \mbox{ on } \partial\Omega, \ |\alpha|\leq 1, \\
\|(\nabla^2 u)^*\|_2<\infty, \end{array}\right.
\hspace{0.5in}
\left\{ \begin{array}{ll} \Delta^2 v=0 & \mbox{ in } \Omega, \\
D^{\alpha}v=D^{\alpha} g & \mbox{ on } \partial\Omega, \ |\alpha|\leq 1, \\
\|(\nabla^2 v)^*\|_2<\infty. \end{array}\right.
$$

\noindent
Note that $\nabla^2 u$ and $\nabla^2 v$ both have nontangential limits
a.e. on $\partial\Omega$.
We now introduce the bilinear form
\begin{equation} \label{defnBilinearForm}
\Lambda[f,g] := \lim_{m\to\infty} \int_{\partial\Omega_m}
\left\{ \pdydx{}{n}\Delta u\cdot v - \Delta u\cdot \pdydx{v}{n} \right\}\,d\sigma
=-\int_\Omega \Delta u\cdot \Delta v\, dx,
\end{equation}
where $\Omega_m\uparrow\Omega$.
Note that $\Lambda[f,g]=\Lambda[g,f]$.
Also, since $\Delta u$ is harmonic in $\Omega$, by Lemma \ref{sobolevLimit}, we have
\begin{equation}\label{bilinear-form-1}
\Lambda[f, g]
=\lim_{m\to\infty}
\int_{\partial\Omega_m}
\left\{ \pdydx{}{n}\Delta u\cdot w - \Delta u\cdot \pdydx{w}{n} \right\}\,d\sigma,
\end{equation}
where $w\in C^2(\Omega\setminus K)$ for some compact subset $K$ of $\Omega$ and
has the property
that $(\nabla w)^*\in L^2(\partial\Omega)$,
and $w=g$, $\nabla w=\nabla g$ a.e. on $\partial\Omega$
in the sense of nontangential convergence.

\begin{prop}\label{regularity-imply-bilinear-1}
Let $2\le p<\infty$ and $\frac{1}{p}+\frac{1}{q}=1$.
Suppose that the $L^{p}$ regularity problem for the biharmonic equation in $\Omega$
is uniquely solvable. Then
\begin{equation} \label{bilinearFormEstimate}
\big|\Lambda[f,g]\big| \leq C
\bigg\{ \|\nabla_t\nabla f \|_p +
\|\nabla f\|_p +\| f\|_p\bigg\}
\bigg\{ \|\nabla g\|_{q} + \|g\|_{q} \bigg\},
\end{equation}
for any functions $f,g$ which are $C^2$ in some neighborhood of $\partial\Omega$.
\end{prop}

\begin{proof}
It follows from the solvability of the $L^p$ regularity problem
that $(\Delta u)^* \in L^p(\partial\Omega)$.
In view of Lemma \ref{sobolevLimit}, we have
$$
\Lambda[f,g]
=\ell (g)-\int_{\partial\Omega} \Delta u\cdot \frac{\partial g}{\partial n}\, d\sigma
$$
 where $\ell=\frac{\partial }{\partial n} \Delta u \in W^{-1,p}(\partial\Omega)$, and
\begin{equation}
\aligned
\big|\Lambda[f,g]\big|
&\le C
\|\Delta u\|_p
\bigg\{ \|g\|_{W^{1,q}(\partial\Omega)} +\|\frac{\partial g}{\partial n}\|_q\bigg\}\\
&\le C
\bigg\{ \|\nabla_t\nabla f\|_p +\| \nabla f\|_p +\| f\|_p\bigg\}
\bigg\{ \|\nabla g\|_q +\| g\|_q\bigg\}.
\endaligned
\end{equation}
\end{proof}

We will show in the next two sections that for any $1<p<\infty$,
the solvability of either the $L^p$ regularity problem or the $L^q$ Dirichlet problem
also implies the bilinear estimate (\ref{bilinearFormEstimate}).
The proofs, however, are much more involved.

\section{$(D)_q$ implies bilinear estimate}

\begin{thm} \label{Dirichlet-imply-bilinear}
 Let $1<p<\infty$ and $\frac{1}{p}+\frac{1}{q}=1$.
If the $L^{q}$ Dirichlet problem for the biharmonic equation in $\Omega$
is uniquely solvable, then the bilinear estimate (\ref{bilinear-estimate}) holds.
\end{thm}

\noindent The proof of Theorem \ref{Dirichlet-imply-bilinear} relies on the following technical lemma.

\begin{lemma} \label{technicalLemma}
    Let $1<q<\infty$.
  Suppose that the $L^{q}$ Dirichlet problem for the biharmonic equation
in $\Omega$ is uniquely solvable.
Also suppose that $u\in C^{\infty}(\Omega)$,
$\nabla^3 u$ exists a.e. on $\partial\Omega$,
and $u=0$ on $\Omega \backslash B(P,r)$ for some $P\in\partial\Omega$.
We further assume that for some $C_0>C_1>10$,
$\Omega \subset B(P,C_0r)$,
$$
B(P,C_1r)\cap \Omega = B(P,C_1r)\cap \big\{(x^\prime,x_d): \ x_d > \eta(x^\prime) \big\},
$$
$\Delta^2 u \in L^{q}(\Omega)$
and
$\mathcal{M}(\nabla D_dD_d u) \in L^{q}(B(P,5r)\cap\partial\Omega)$.
 Then,
 \begin{equation}
\aligned
\int_{I(2r)} |\mathcal{M}(\nabla^3 u)|^{q}\,d\sigma &
\leq
 C\int_{\partial\Omega}
\left|\nabla D_dD_du \right|^{q}\,d\sigma
+ C \int_{\partial\Omega} \left| D_dD_d u \right|^{q}\,d\sigma  \\
& \qquad\qquad+ C\int_{\Omega} |\Delta^2 u|^{q}\,dx
+ C\int_{K} |\nabla^3 u|^{q}\,dx,
\endaligned
\end{equation}
where $I(t)=B(P,t)\cap \partial\Omega$ and $K$ is a compact subset of $\Omega$.
\end{lemma}

\begin{proof}
    We may assume that $P=0$.
Let $$
\widetilde{u}=u-\Gamma*(\Delta^2 u)
=u(x)-\int_\Omega \Gamma(x-y)\cdot \Delta^2 u(y)\, dy,
$$
where $\Gamma (x)$ is the fundamental solution of the biharmonic equation
with pole at the origin.
Then $\Delta^2 \widetilde{u}=0$ in $\Omega$
and
    \begin{equation} \label{techLemmaEqn1}\int_{I(2r)}
|\mathcal{M}(\nabla^3 u)|^{q}\,d\sigma \leq C\int_{I(2r)}
|\mathcal{M}(\nabla^3 \widetilde{u})|^{q}\,d\sigma
+ C\int_{I(2r)}
|\mathcal{M}(\nabla^3 \Gamma* (\Delta^2 u))|^{q}\,d\sigma.
\end{equation}
We now estimate both terms on the right hand side of (\ref{techLemmaEqn1}),
 beginning with the second term.
Using Lemma \ref{radialMaxLemma}
 as well as the fractional and singular integral estimates, we obtain
    \begin{equation}\aligned
        \int_{I(2r)} |\mathcal{M}(\nabla^3 \Gamma* (\Delta^2 u ))|^{q}\,d\sigma
&\leq C\int_{\Omega} |\nabla^4 \Gamma* (\Delta^2 u)|^{q}\,dx
+ C\int_{\Omega} |\nabla^3 \Gamma* (\Delta^2 u)|^{q}\,dx \\
        &\leq C \int_{\Omega} |\Delta^2 u|^{q}\,dx. \label{techLemmaEqn1b}
\endaligned
    \end{equation}

To estimate the first term in the right hand side of (\ref{techLemmaEqn1}),
we use the square function estimate (\ref{square-dominate}) to obtain
    \begin{equation}\aligned
        \int_{I(2r)}
|\mathcal{M}(\nabla^3 \widetilde{u})|^{q}\,d\sigma
&\leq C \int_{\partial\Omega} |(\nabla^3 \widetilde{u})^*|^{q}\,d\sigma \\
        &\leq C\int_{\partial\Omega} |S(\nabla^3 \widetilde{u})|^{q}\,d\sigma
+ C\sup_{K_1} |\nabla^3 \widetilde{u}|^{q} \label{techLemmaEqn2},
\endaligned
    \end{equation}
    \noindent where $K_1$ is a compact subset of $\Omega$.
  With the well known interior estimates for biharmonic functions,
we can handle the second term in (\ref{techLemmaEqn2}) in the following way:
 \begin{equation} \label{techLemmaEqn3}
 \sup_{K_1} |\nabla^3 \widetilde{u}|^{q}
\leq C\int_{K_2} |\nabla^3 \widetilde{u}|^{q}\,dx
\leq  C\int_{\Omega} |\Delta^2 u|^{q}\,dx
+C\int_{K_2} |\nabla^3 u|^{q}\,dx
\end{equation}
where $K_2\supset K_1$ is a compact subset of $\Omega$.

    Next we will estimate the term involving the square function
in (\ref{techLemmaEqn2}) with the following observations,
    \begin{eqnarray}
        S(\nabla^3 \widetilde{u})
&=& S(\nabla^3 \Gamma* (\Delta^2u)) \leq C\int_{\Omega} |\Delta^2 u|\,dx \hspace{0.25in}
\mbox{ on } \partial\Omega \backslash B(0,3r), \label{techLemmaEqn4}\\
        S(\nabla^3 \widetilde{u})
&\leq& C\bigg\{ \widetilde{S}\left(\nabla D_dD_d
\widetilde{u}\right) + \sup_{K_3} |\nabla^3 \widetilde{u}|\bigg\}
\hspace{0.25in} \mbox{ on } B(0,3r) \cap \partial\Omega, \label{techLemmaEqn5}
    \end{eqnarray}
    \noindent where $K_3\supset K_2$ is a compact subset of $\Omega$
and $\widetilde{S}(w)$ denotes a square function of $w$,
defined by using a regular family of nontangential cones
 which are slightly larger than the ones used for $S(w)$.
 Estimate (\ref{techLemmaEqn4})
 is a simple consequence of the fact that $u=0$ on $\Omega \backslash B(0,r)$, while
 estimate (\ref{techLemmaEqn5}) follows by applying Lemma \ref{squareFunctionProp}
twice.

    Now, using estimates (\ref{techLemmaEqn4}) and (\ref{techLemmaEqn5}) we obtain
    \begin{equation}\aligned
        \int_{\partial\Omega} |S(\nabla^3 \widetilde{u})|^{q}\,d\sigma
&=
\int_{I(3r)} |S(\nabla^3 \widetilde{u})|^{q}\,d\sigma
+ \int_{\partial\Omega \backslash I(3r)} |S(\nabla^3 \widetilde{u})|^{q}\,d\sigma  \\
        &\leq C \left\{ \int_{I(3r)} \left| \widetilde{S} \left(\nabla D_dD_d
 \widetilde{u}\right)\right|^{q}\,d\sigma + \int_{\Omega} |\Delta^2 u|^{q}\,dx
+ \sup_{K_3} |\nabla^3 \widetilde{u}|^{q}\right\}\\
        &\leq C\left\{ \int_{\partial\Omega} \left|
\left(\nabla D_dD_d \widetilde{u} \right)^*\right|^{q}\,d\sigma
+ \int_{\Omega} |\Delta^2 u|^{q}\,dx + \int_{K_4} |\nabla^3 \widetilde{u}|^{q}\,dx\right\}\\
&\leq C\left\{ \int_{\partial\Omega} \left|
\left(\nabla D_dD_d \widetilde{u} \right)^*\right|^{q}\,d\sigma
+ \int_{\Omega} |\Delta^2 u|^{q}\,dx + \int_{K_4} |\nabla^3 {u}|^{q}\,dx\right\}
\endaligned
\label{techLemmaEqn6}
    \end{equation}
where we have used the square function estimate (\ref{maximal-dominate}) in the
second inequality.

Note that $ D_d D_d \widetilde{u}$ is biharmonic in $\Omega$. We claim that
\begin{equation}\label{claim}
(\nabla D_d D_d \widetilde{u})^*\in L^q(\partial\Omega).
\end{equation}
Assume the claim for a moment.  We may then use the solvability
of the $L^{q}$ Dirichlet problem
 to estimate the first term in the right hand side of (\ref{techLemmaEqn6}).
This, together with Lemma \ref{radialMaxLemma}, gives

   \begin{equation}\aligned
      &   \int_{\partial\Omega} |(\nabla D_dD_d \widetilde{u})^*|^{q}\,d\sigma \\
&\leq C\int_{\partial\Omega} |\nabla D_dD_d \widetilde{u}|^{q}\,d\sigma
+ C\int_{\partial\Omega} |D_d D_d \widetilde{u}|^{q}\,d\sigma  \\
        & \leq  C\int_{\partial\Omega} |\nabla D_dD_d u|^{q}\,d\sigma
+ C\int_{\partial\Omega} |\nabla D_d D_d (\Gamma*(\Delta^2 u))|^q\,d\sigma \\
&  \qquad\qquad + C\int_{\partial\Omega} |D_d D_d u|^{q}\,d\sigma
+ C\int_{\partial\Omega} |D_d D_d (\Gamma*(\Delta^2 u))|^{q}\,d\sigma \\
                &\leq  C\int_{\partial\Omega} |\nabla D_dD_d u|^{q}\,d\sigma
+ C\int_{\partial\Omega} |D_d D_d u|^{q}
+ C\int_{\Omega} |\nabla^4 \Gamma* (\Delta^2 u)|^{q}\,dx \\
&  \qquad\qquad+C\int_{\Omega} |\nabla^3 \Gamma* (\Delta^2 u)|^{q}\,dx
+ C\int_{\Omega} |\nabla^2 \Gamma* (\Delta^2 u)|^{q}\,dx \\
        &\leq  C\left\{ \int_{\partial\Omega} |\nabla D_dD_d u|^{q}\,d\sigma
+ \int_{\partial\Omega} |D_d D_d u|^{q}\,d\sigma
+ \int_{\Omega} |\Delta^2 u|^{q}\,dx \right\}. \label{techLemmaEqn7}
\endaligned
   \end{equation}
By combining estimates (\ref{techLemmaEqn1}),
(\ref{techLemmaEqn1b}),
(\ref{techLemmaEqn2}),
(\ref{techLemmaEqn3}),
(\ref{techLemmaEqn6}) and (\ref{techLemmaEqn7}), we obtain the desired estimate.

Finally we need to prove the claim (\ref{claim}).
Since $\nabla D_d D_d \widetilde{u}$ is bounded in $\Omega\setminus B(0,2r)$, we only need to show that
$(\nabla D_dD_d \widetilde{u})^*\in L^q (I(3r))$. Note that if $\Delta^2 w=0$ in $\Omega$, we may use
interior estimates to show that
$$
(w)^*\le C M_{\partial\Omega} (\mathcal{M}(w))\quad\quad\text{ on } I(3r),
$$
where $M_{\partial\Omega}$ denotes
the Hardy-Littlewood maximal operator on $\partial\Omega$.
Hence it is enough to prove that
$\mathcal{M}(\nabla D_d D_d \widetilde{u})\in L^q(I(5r))$.
But this is easy as
$\mathcal{M}(\nabla D_dD_d u)\in L^q(I(5r))$ by the assumption and the desired estimate for
$\mathcal{M}(\nabla D_dD_d \Gamma*(\Delta^2 u))$ follows directly from
Lemma \ref{radialMaxLemma}.
The proof of Lemma \ref{technicalLemma} is now complete.
\end{proof}

We now proceed to the proof of Theorem \ref{Dirichlet-imply-bilinear}.

\begin{proof}(of Theorem \ref{Dirichlet-imply-bilinear})
    It suffices to consider the case $\mbox{supp}(g)\subset B(P_0,r)$
where $P_0\in\partial\Omega$ and $B(P_0,C_1r)\cap \partial\Omega$
is given by the graph of a Lipschitz function after a possible rotation.
  For otherwise write $g=\sum_{j=1}^{M} g\varphi_j$
where $\varphi_j \in C^{\infty}_0(\rn{d})$
and $\sum \varphi_j=1$ on $\partial\Omega$.
Then, $v=\sum v_j$ where $v_j$ is the solution of the $L^2$ regularity problem
with boundary data $D^\alpha(g\varphi_j)$
and we have
    \begin{equation}\aligned
        \big|\Lambda[f,g]\big| &\leq
\sum_j |\Lambda[f,\varphi_j g]| \\
&\leq
C\sum_j \bigg\{ \|\nabla_t\nabla (\varphi_j g) \|_p
+\|\nabla (\varphi_j g)\|_p +\| \varphi_j g\|_p\bigg\}
 \bigg\{ \|\nabla f \|_q + \|f\|_q\bigg\}
\\ &\leq C \bigg\{ \|\nabla_t\nabla g\|_p +\|\nabla g\|_p +\| g\|_p\bigg\}
 \bigg\{\|\nabla f\|_q + \|f\|_q \bigg\}.
    \endaligned
\end{equation}

    Assume $P_0=0$ and choose $\psi \in C_0^{\infty}(B(0,4r))$
such that $\psi=1$ on $B(0,3r)$ and
$|D^\alpha \psi| \leq Cr^{-|\alpha|}$ for $|\alpha|\le 4$.  Now define
    \begin{eqnarray*}
        u_{-1}(x^\prime,x_d) &=& - \int_{x_d}^{\infty} u(x',s)\psi(x',s)\,ds, \\
        u_{-2}(x^\prime,x_d) &=&
\int_{x_d}^{\infty} \int_{t}^{\infty} u(x^\prime,s)\psi(x^\prime,s)\,dsdt.
    \end{eqnarray*}
    Note that $D_du_{-2}=u_{-1}$ and $D_dD_du_{-2}=u\psi$ in $\Omega$
and in particular $D_dD_d u_{-2}=u$ on $B(0,3r)\cap \Omega$.  Also, note that
\begin{eqnarray*}
	\Delta^2 u_{-1} &=& - \int_{x_d}^{\infty} \Delta^2 \{ u(x',s)\psi(x',s)\}\,ds   \\
	&=& -\int_{x_d}^{\infty} \big\{
 4D_i\psi\cdot D_i\Delta u 
+ 2\Delta\psi \cdot \Delta u 
+ 4D_iD_ku\cdot D_iD_k \psi  \\ 
&& \qquad\qquad
\qquad\qquad  + 4D_iu\cdot D_i \Delta\psi + u\cdot \Delta^2 \psi\big\}\,ds
\end{eqnarray*}
\noindent where we've used the fact that $\Delta^2 u=0$ in $\Omega$.  
Since supp$(\nabla\psi)\subset B(0,4r)\setminus B(0,3r)$,
 it follows that on $B(0,2r)\cap \Omega$, 
\begin{equation} \label{firstPrimitaveOperatorEst} 
|\Delta^2 u_{-1}|\leq C\sup_{K} |u|,\end{equation} where $K\subset\subset \Omega$ is compact.

Let $w=v\varphi$ where $\varphi \in C_0^\infty (B(0,2r)$ and $\varphi=1$ in $B(0,r)$.
Then,
    \begin{equation}\label{primitivesAdded}\aligned
        &\int_{\partial\Omega_m} w \cdot \pdydx{}{n} \Delta u\,d\sigma
- \int_{\partial\Omega_m} \pdydx{w}{n}\cdot \Delta u\,d\sigma  \\
        &= \int_{\partial\Omega_m} w\cdot n_kD_kD_iD_iD_dD_d u_{-2}\,d\sigma
- \int_{\partial\Omega_m} n_kD_k w
\cdot D_iD_iD_dD_d u_{-2}\,d\sigma
\endaligned
    \end{equation}
where $\Omega_m\uparrow \Omega$.

   We begin by working with the second term in the right hand side of
 (\ref{primitivesAdded}).  Note that
    \begin{eqnarray}
        \lefteqn{\int_{\partial\Omega_m} n_kD_k w\cdot
D_iD_iD_dD_d u_{-2}\,d\sigma} \nonumber\\
&=& \int_{\partial\Omega_m} D_k w\cdot
(n_kD_d-n_dD_k) D_iD_iD_du_{-2}\,d\sigma +
\int_{\partial\Omega_m} D_k w\cdot n_dD_kD_iD_i D_d u_{-2}\,d\sigma \nonumber\\
\hspace{-0.1in}&=& -\int_{\partial\Omega_m} (n_kD_d-n_dD_k)D_k w
\cdot D_iD_iD_du_{-2}\,d\sigma \nonumber \\
&& \qquad + \int_{\partial\Omega_m} D_k w\cdot n_dD_kD_iD_iD_du_{-2}\,d\sigma.
 \label{tangentialCreated1}
    \end{eqnarray}

 We now consider the first term in the right hand side of
 (\ref{primitivesAdded}). Using integration by parts we obtain
    \begin{equation} \label{tangentialCreated2}
\aligned
        &\int_{\partial\Omega_m} w\cdot
n_kD_kD_iD_iD_dD_d u_{-2}\,d\sigma  \\
&= \int_{\partial\Omega_m} w\cdot
(n_kD_d-n_dD_k)D_kD_iD_iD_du_{-2}\,d\sigma+\int_{\partial\Omega_m} w\cdot n_dD_kD_kD_iD_iD_du_{-2}\,d\sigma\\
        &= -\int_{\partial\Omega_m} (n_kD_d-n_dD_k) w
\cdot D_kD_iD_iD_du_{-2}\,d\sigma+\int_{\partial\Omega_m} w\cdot n_dD_kD_kD_iD_iD_du_{-2}\,d\sigma.
\endaligned
    \end{equation}
 By combining equations (\ref{primitivesAdded}), (\ref{tangentialCreated1}),
and (\ref{tangentialCreated2}) we obtain
$$\aligned
       & \int_{\partial\Omega_m} w\cdot
n_kD_k D_iD_iD_dD_d u_{-2}\,d\sigma -
\int_{\partial\Omega_j} n_kD_k w\cdot D_iD_iD_dD_d u_{-2}\,d\sigma \\
&=\int_{\partial\Omega_m} (n_dD_k-n_kD_d)w\cdot D_kD_iD_iD_du_{-2}\,d\sigma
 + \int_{\partial\Omega_m} (n_kD_d-n_dD_k)D_k w\cdot D_iD_iD_du_{-2}\,d\sigma \\
 & \qquad \qquad- \int_{\partial\Omega_m} D_k w\cdot n_dD_kD_iD_iD_du_{-2}\,d\sigma +\int_{\partial\Omega_m} w\cdot n_dD_kD_kD_iD_iD_du_{-2}\,d\sigma \\
        &= - \int_{\partial\Omega_m} n_k D_d w\cdot D_kD_iD_iD_du_{-2}\,d\sigma
+ \int_{\partial\Omega_m} (n_kD_d-n_dD_k)D_k w\cdot D_iD_iD_d u_{-2}\,d\sigma \\
& \qquad\qquad +\int_{\partial\Omega_m} w\cdot n_dD_kD_kD_iD_iD_du_{-2}\,d\sigma \\
        &=\int_{\partial\Omega_j} D_d w\cdot (n_dD_k-n_kD_d)D_kD_iD_iu_{-2}\,d\sigma
+ \int_{\partial\Omega_m} (n_kD_d-n_dD_k)D_k w\cdot D_iD_iD_du_{-2}\,d\sigma \\
& \qquad\qquad +\int_{\partial\Omega_m} w\cdot n_dD_kD_kD_iD_iD_du_{-2}\,d\sigma \\
        &= \int_{\partial\Omega_m} (n_kD_d-n_dD_k)D_d w\cdot D_kD_iD_iu_{-2}\,d\sigma
 + \int_{\partial\Omega_m} (n_kD_d-n_dD_k)D_k w\cdot D_iD_i u_{-1}\,d\sigma. \\
 &\qquad\qquad +\int_{\partial\Omega_m} w\cdot n_dD_kD_kD_iD_iD_du_{-2}\,d\sigma
\endaligned
$$
    Therefore we have proved that
    \begin{eqnarray}
        \lefteqn{\int_{\partial\Omega_m} w\cdot
\pdydx{}{n}\Delta u\,d\sigma
- \int_{\partial\Omega_m} \pdydx{w}{n}\cdot
\Delta u\,d\sigma} \nonumber \\
&& = \int_{\partial\Omega_m} (n_kD_d-n_dD_k)D_d w
\cdot D_kD_iD_iu_{-2}\,d\sigma
\label{takeLimits} \\
&& \qquad\qquad + \int_{\partial\Omega_m} (n_kD_d-n_dD_k)D_k w
\cdot D_iD_iu_{-1}\,d\sigma +\int_{\partial\Omega_m} w\cdot n_dD_kD_kD_iD_iD_du_{-2}\,d\sigma. \nonumber
    \end{eqnarray}

    We now let $m\to \infty$.
Note that the left hand side of (\ref{takeLimits}) becomes $\Lambda[f,g]$
in view of Lemma \ref{sobolevLimit}.
The second term on the right hand side has a limit
because both $\nabla^2 u$ and $\nabla^2 v$
have nontangential limits a.e. on $\partial\Omega$ and their nontangential maximal functions are in $L^2$ since $u$ and $v$ are $L^2$ solutions of the
regularity problem.
The first term on the right hand side of (\ref{takeLimits}) can be handled by Lemmas \ref{nonTangLimitLemma}
and \ref{technicalLemma}.
Recall that $w$ is supported in the ball $B(0,2r)$.  Thus,
Lemma \ref{technicalLemma} implies that $\mathcal{M}(\nabla^3 u_{-2})\in L^q(I(8r))$.
It then follows from Lemma \ref{nonTangLimitLemma} that $\nabla^3 u_{-2}$ has nontangential limits
a.e. on $I(2r)$ and $(\nabla^3 u_{-2})^*\in L^q (I(2r))$.
  The final term on the right hand side of (\ref{takeLimits}) has a limit because $v$ has nontangential limits a.e. on $\partial\Omega$ and because of estimate (\ref{firstPrimitaveOperatorEst}).
Thus, letting $m\to \infty$ in the equation (\ref{takeLimits}) leads to
    \begin{equation}\label{almost-there}
        \big| \Lambda[f,g]\big| \leq C \left\{  \|\nabla_t \nabla (g\varphi) \|_p \|\nabla^3 u_{-2} \|_q + \|g\varphi\|_p\sup_{K} |u| \right\}.
    \end{equation}

Note that since the $L^q$ Dirichlet problem in $\Omega$ is solvable, 
we have that 
\begin{eqnarray} \sup_{K} |u| \leq C\|(u)^*\|_q \leq C\left\{ \|\nabla f\|_q + \|f\|_q\right\}. \label{supEstSolvability} \end{eqnarray}  Thus, we only need to estimate $\|\nabla^3 u_{-2}\|_q$.  To do this we apply Lemma \ref{technicalLemma} to obtain
    \begin{equation}\aligned
        \int_{\partial\Omega} |\nabla^3 u_{-2}|^q\,d\sigma
&\leq C\int_{\partial\Omega} |\nabla (u\psi)|^q\,d\sigma
+ C\int_{\partial\Omega} |u\psi|^q\,d\sigma  \\
& \qquad\qquad +C\int_K |\nabla^3 u_{-2}|^q\,dx
+C\int_{\Omega} |\Delta^2 u_{-2}|^q\,dx.
\endaligned \label{lowerOrderTerms}
    \end{equation}
Note that by the interior estimates and the solvability of the $L^q$ Dirichlet problem we have
    \begin{equation} \label{lowerOrderEst1}
      \int_{K} |\nabla^3 u_{-2}|^q\,dx \leq C\sup_{\widetilde{K}} |u|^q
\leq C\left\{ \int_{\partial\Omega} |f|^q\,d\sigma + \int_{\partial\Omega} |\nabla f|^q\,d\sigma\right\}.
    \end{equation}

 To estimate the last term in (\ref{lowerOrderTerms}) we use Hardy's inequality twice to obtain
    \begin{equation}\aligned
        \int_{\Omega} |\Delta^2 u_{-2}|^q \, dx
&\leq C\int_{ \Omega} |\Delta^2 u_{-1}|^q
\left\{\mbox{dist}(x,\partial\Omega)\right\}^q\,dx  \\
        &\leq C
\int_{\Omega}
|\Delta^2 (u\psi)|^q\left\{ \mbox{dist}(x,\partial\Omega)\right\}^{2q}\,dx  \\
        &\leq C\int_\Omega
 \left\{ |\nabla^3 u|^q + |\nabla^2 u|^q + |\nabla u|^q + |u|^q\right\}
\left\{ \mbox{dist}(x,\partial\Omega)\right\}^{2q}\,dx  \\
        &\leq C\int_{\partial\Omega} \big\{ |(\nabla u)^*|^{q} +|(u)^*|^q\big\} \,d\sigma \\
        &\leq C\int_{\partial\Omega} \big\{ |\nabla f|^{q}
+  |f|^q \big\} \,d\sigma.
\endaligned
\label{lowerOrderEst3}
    \end{equation}
    \noindent Thus, combining estimates (\ref{lowerOrderTerms}),
(\ref{lowerOrderEst1}) and (\ref{lowerOrderEst3}) we obtain
    \begin{equation} \label{lowerOrderEstFinal}
        \|\nabla^3 u_{-2}\|_q \leq C \bigg\{ \|\nabla f\|_q +\|f\|_q \bigg\}.
    \end{equation}

  In view of (\ref{almost-there}), (\ref{supEstSolvability}) and (\ref{lowerOrderEstFinal}), we may conclude that
 $$
\big|\Lambda[f,g]\big|
\leq C \bigg\{ \|\nabla_t\nabla g \|_p +\| \nabla g\|_p +\| g\|_p\bigg\}
 \bigg\{\|\nabla f\|_q + \|f\|_q \bigg\}.
$$
Since $\Lambda[f,g]=\Lambda[g,f]$, this completes the proof.
\end{proof}

\section{$(R)_p$ implies bilinear estimate}

\begin{thm} \label{regularity-bilinear}
 Let $1<p<\infty$ and $\frac{1}{p}+\frac{1}{q}=1$.
If the $L^{p}$ regularity problem for the biharmonic equation in $\Omega$
is uniquely solvable, then the bilinear
estimate (\ref{bilinear-estimate}) holds.
\end{thm}

The proof of Theorem \ref{regularity-bilinear} is similar to (and slightly simpler than)
that of Theorem \ref{Dirichlet-imply-bilinear}.

\begin{lemma} \label{technicalLemma-1}
    Let $1<p<\infty$.
  Suppose that the $L^{p}$ regularity problem for the biharmonic equation
in $\Omega$ is uniquely solvable.
Also suppose that $u\in C^{\infty}(\Omega)$,
$\nabla^3 u$ exists a.e. on $\partial\Omega$,
and $u=0$ on $\Omega \backslash B(P,r)$ for some $P\in\partial\Omega$.
We further assume that for some $C_0>C_1>10$,
$\Omega \subset B(P,C_0r)$,
$$
B(P,C_1r)\cap \Omega = B(P,C_1r)\cap \big\{(x^\prime,x_d): \ x_d > \eta(x^\prime) \big\},
$$
$\Delta^2 u \in L^p(\Omega)$
and
$\mathcal{M}(\nabla^2 D_d u) \in L^p(B(P,5r)\cap\partial\Omega)$.
 Then,
 \begin{equation}\label{estimate-6-1}
\aligned
\int_{I(2r)} |\mathcal{M}(\nabla^3 u)|^{p}\,d\sigma &
\leq
 C\int_{\partial\Omega}
\left|\nabla^2 D_d u \right|^p\,d\sigma
+ C \int_{\partial\Omega} \big\{ |\nabla D_d u|^p +|D_d u|^p \big\}\,d\sigma
  \\
& \qquad\qquad+ C\int_{\Omega} |\Delta^2 u|^p\,dx
+ C\int_{K} |\nabla^3 u|^p\,dx,
\endaligned
\end{equation}
where $I(t)=B(P,t)\cap \partial\Omega$ and $K$ is a compact subset of $\Omega$.
\end{lemma}

\begin{proof}
We will only give a sketch of the proof, as it is similar to that of Lemma \ref{technicalLemma}.
Let $\widetilde{u}=u-\Gamma*(\Delta^2 u)$. An inspection of
the proof of Lemma \ref{technicalLemma} shows that it suffices to
bound $\| (\nabla^2 D_d \widetilde{u})^*\|_p^p$ by the right hand side
of (\ref{estimate-6-1}).
To this end we first note that $D_d \widetilde{u}$ is biharmonic in $\Omega$.
Also, since $\Delta^2 u\in L^p(\Omega)$ and $\mathcal{M}(\nabla^2 D_d u)\in L^p(I(5r))$
by assumption,
it follows from the same argument as in the proof of Lemma \ref{technicalLemma} that
\begin{equation}\label{claim-1}
(\nabla^2 D_d \widetilde{u})^*\in L^p(\partial\Omega).
\end{equation}
As a result, we may use the solvability of the $L^p$
regularity problem in $\Omega$  as well as Lemma \ref{radialMaxLemma}
to obtain
\begin{equation}
\aligned
& \int_{\partial\Omega} |(\nabla^2 D_d \widetilde{u})^*|^p\, d\sigma\\
&\le C \int_{\partial\Omega} \big\{
|\nabla^2 D_d \widetilde{u}|^p+ |\nabla D_d \widetilde{u}|^p
+|D_d \widetilde{u}|^p\big\} \, d\sigma\\
&\le C \int_{\partial\Omega} |\nabla^2 D_d {u}|^p\, d\sigma
+C \int_{\partial\Omega} \big\{ |\nabla D_d {u}|^p
+|D_d {u}|^p\big\} \, d\sigma
+C \int_\Omega |\Delta^2 u|^p \, dx.
\endaligned
\end{equation}
This finishes the proof.
\end{proof}

We now are in a position to give the proof of Theorem \ref{regularity-bilinear}.

\begin{proof} (of Theorem \ref{regularity-bilinear})
By a partition of unity as well as rotation and translation
we may assume that supp$(g)\subset B(0,r)$, $0\in  \partial\Omega$,
and $B(0, C_1 r)\cap\partial\Omega$
is given by the graph of a Lipschitz function. Since the integral of
$\Delta u\cdot \frac{\partial v}{\partial n}$ on $\partial\Omega$
 is clearly bounded by the right hand side
of (\ref{bilinear-estimate}), it suffices
to bound
\begin{equation}
L=\bigg| \lim_{m\to \infty}
\int_{\partial\Omega_m} w\cdot \frac{\partial}{\partial n} \Delta u\, d\sigma\bigg|,
\end{equation}
where $w=v\varphi$ with $\varphi\in C_0^\infty(B(0,2r))$ and $\varphi=1$ in $B(0,r)$.

To this end we define $u_{-1}$ as in the proof of Theorem \ref{Dirichlet-imply-bilinear} and
write
\begin{equation}
w\cdot \frac{\partial}{\partial n}\Delta u =
w\cdot (n_j D_d -n_d D_j)D_j \Delta u_{-1} + w\cdot n_d \Delta^2 u_{-1},
\end{equation}
It follows that
\begin{equation}\label{L-1}
\aligned
L & =\bigg| \lim_{m\to \infty}
\left[\int_{\partial\Omega_m}
(n_j D_d -n_d D_j)w \cdot D_j \Delta u_{-1}\, d\sigma + \int_{\partial\Omega_m} w\cdot n_d \Delta^2 u_{-1} \right] \bigg|\\
&
\le C \, \| \nabla_t (g\varphi)\|_{L^q(\partial\Omega)}
\| \mathcal{M} (\nabla^3 u_{-1})\|_{L^p(I(2r))} + C\|g\varphi\|_{L^q(\partial\Omega)}\cdot \sup_K |u|,
\endaligned
\end{equation}
where $K\subset\subset \Omega$ is a compact set.  First, note that since the $L^p$ regularity problem is uniquely solvable we have \begin{equation} \label{L-3} \sup_K |u| \leq C\|(u)^*\|_p \leq C\left\{ \|\nabla_t\nabla f\|_p + \|\nabla f\|_p + \|f\|_p\right\}. \end{equation}  Also, note that since $(R)_p$ in $\Omega$ is uniquely solvable,
we have $(\nabla^2 u)^*\in L^p(\partial\Omega)$.
By Hardy inequality, this implies that $\Delta^2 u_{-1}\in L^p(\Omega)$ and
$\mathcal{M}(\nabla^2 D_d u_{-1})\in L^p(I(5r))$.
Hence we may apply Lemma \ref{technicalLemma-1} to the function $u_{-1}$.
This leads to
\begin{equation}\label{L-2}
\aligned
\int_{I(2r)} |\mathcal{M}(\nabla^3 u_{-1})|^p\, d\sigma
&\le C\int_{\partial\Omega}
\bigg\{ |\nabla^2(u\psi)|^p +|\nabla(u\psi)|^p +|u\psi|^p\bigg\}\, d\sigma\\
&\quad\quad\quad\quad
+C\int_\Omega |\Delta^2 u_{-1}|^p \, dx
+C \int_K |\nabla^3 u_{-1}|^p\, dx\\
&
\le C \int_{\partial\Omega}
\bigg\{ |\nabla^2 u|^p +|\nabla u|^p +|u|^p\bigg\} d\sigma\\
& \quad\quad\quad
+C \int_\Omega |\Delta^2 (u\psi)|^p  \big\{ \text{dist}(x, \partial\Omega)\big\}^p\, dx
+C \sup_{K_1} |u|^p\\
&\le C
\int_{\partial\Omega}
\bigg\{ |(\nabla^2 u)^*|^p
+|(\nabla u)^*|^p
+|(u)^*|^p\bigg\} \, d\sigma\\
&\le C
\int_{\partial\Omega}
\bigg\{ |\nabla_t \nabla f|^p +|\nabla f|^p +|f|^p\bigg\}\, d\sigma,
\endaligned
\end{equation}
where we have used Hardy's inequality in the second inequality
and the $L^p$ regularity estimate in the last inequality.
The desired estimate for $L$ now follows from (\ref{L-1}), (\ref{L-3}), and (\ref{L-2}).
\end{proof}

\section{Bilinear estimate implies $(R)_p$ and $(D)_q$}

In the previous two sections
 we proved that both the solvability of $(D)_q$ and the solvability of $(R)_p$
imply the bilinear estimate (\ref{bilinear-estimate}).
We will see in this section that the converse is also true.

\begin{thm}\label{bilinear-imply-regularity}
Let $1<p<\infty$ and $\frac{1}{p}+\frac{1}{q}=1$.
 Suppose the bilinear estimate (\ref{bilinear-estimate})
holds for any $f, g\in C_0^\infty(\rn{d})$.
Then the $L^p$ regularity problem is uniquely solvable.
\end{thm}

\begin{proof}
To establish the existence in $(R)_p$,
it suffices to show that if $g\in C_0^\infty(\rn{d})$ and
$u$ is the unique solution of the $L^2$ regularity problem with boundary
data $D^\alpha g$, then estimate (\ref{regularity-estimate}) holds.
The existence of solutions with data in
$\mbox{\emph{WA}}^{2,p}(\partial\Omega)$
follows from this by a standard approximation argument
(see e.g. \cite{kilty:higherOrderRegularity} for the case of second order
elliptic systems).

Let $\Omega_m\uparrow\Omega$.
It follows from the Green's identity (\ref{Green's-identity}) that
\begin{equation}\label{Green's-representation}
u(x)=\int_{\partial\Omega_m} \left\{
\frac{\partial}{\partial n} \Delta \Gamma^x \cdot u
-\Delta \Gamma^x \cdot \frac{\partial u}{\partial n}\right\} d\sigma
+\int_{\partial\Omega_m}
\left\{ \frac{\partial \Gamma^x}{\partial n}
\cdot \Delta u
-\Gamma^x \cdot \frac{\partial}{\partial n}\Delta u\right\} d\sigma.
\end{equation}
This implies that
\begin{equation}\label{Green's-representation-1}
\aligned
 D_jD_k u(x)= &
\int_{\partial\Omega_m} \left\{
\frac{\partial}{\partial n} \Delta D_jD_k \Gamma^x \cdot u
-\Delta D_jD_k \Gamma^x \cdot \frac{\partial u}{\partial n}\right\} d\sigma\\
&\quad\quad\quad
+\int_{\partial\Omega_m}
\left\{ \frac{\partial }{\partial n} D_jD_k \Gamma^x
\cdot \Delta u
-D_jD_k \Gamma^x \cdot \frac{\partial}{\partial n}\Delta u\right\} d\sigma.
\endaligned
\end{equation}

To deal with the first term in the right hand side of (\ref{Green's-representation-1}),
we use the following identity,
\begin{equation}\label{identity}
\aligned
& n_\ell D_\ell \Delta D_j D_k \Gamma^x \cdot u -\Delta D_j D_k \Gamma^x \cdot n_\ell D_\ell u\\
&=(n_\ell D_j-n_j D_\ell) \big\{ D_\ell \Delta D_k \Gamma^x\cdot u\big\}
-(n_\ell D_k-n_k D_\ell) \big\{ D_\ell \Delta \Gamma^x \cdot D_j u\big\}\\
&\quad\quad\quad
+D_\ell \Delta \Gamma^x \cdot (n_\ell D_k -n_kD_\ell)D_j u
+(n_jD_\ell -n_\ell D_j) \big\{ \Delta D_k \Gamma^x \cdot D_\ell u\big\}\\
&\quad\quad\quad
-\Delta D_k \Gamma^x\cdot (n_j D_\ell -n_\ell D_j) D_\ell u,
\endaligned
\end{equation}
which may be verified by a direct computation, using the fact that $\Delta^2 \Gamma^x=0$ in $\rn{d}\setminus
\{ x\}$.

In view of (\ref{identity}), we may use integration by parts to obtain
\begin{equation}\label{Green's-representation-2}
\aligned
D_jD_k u(x)= &
\int_{\partial \Omega_m} D_\ell \Delta \Gamma^x \cdot (n_\ell D_k -n_k D_\ell) D_j u\, d\sigma\\
&\quad \quad\quad\quad
-\int_{\partial\Omega_m}
\Delta D_k \Gamma^x \cdot (n_j D_\ell -n_\ell D_j)D_\ell u \, d\sigma\\
&
+\int_{\partial\Omega_m} \left\{ \frac{\partial}{\partial n}
D_jD_k \Gamma^x \cdot \Delta u
-D_jD_k \Gamma^x \cdot \frac{\partial}{\partial n}\Delta u\right\} \, d\sigma.
\endaligned
\end{equation}
We now let $m\to \infty$ in (\ref{Green's-representation-2}).
 It follows from (\ref{bilinear-form-1})
that
\begin{equation}\label{Green's-representation-3}
D_jD_k u(x)= I_{jk}(x)-\Lambda[g, D_jD_k \Gamma^x]
\end{equation}
where
$$
I_{jk} (x)=
\int_{\partial \Omega} D_\ell \Delta \Gamma^x \cdot (n_\ell D_k -n_k D_\ell) D_j u\, d\sigma
-\int_{\partial\Omega}
\Delta D_k \Gamma^x \cdot (n_j D_\ell -n_\ell D_j)D_\ell u \, d\sigma
$$
and $\Lambda[\, , \, ]$ is the bilinear form defined by (\ref{defnBilinearForm}).
Clearly,
\begin{equation}\label{estimate of Ijk}
\|(I_{jk})^*\|_p \le C\, \|\nabla_t \nabla u\|_p.
\end{equation}

To handle the second term in the right hand side of (\ref{Green's-representation-3}),
we  observe that the bilinear estimate (\ref{bilinear-estimate}) implies that
$$
T_g (\dot{f}) =\Lambda[g, f]
$$
defines a bounded linear functional on $\mbox{\emph{WA}}^{1,q}(\partial\Omega)$
which can be identified with $W^{1,q}(\partial\Omega)\times L^q(\partial\Omega)$
by the map $(f_0, f_1, \dots, f_d)\to (f_0, n_1 f_1 +\cdots +n_d f_d)$.
It follows that there exist functions
$G_0$, $G_{s\ell}$, $H$ in $L^p(\partial\Omega)$ such that
\begin{equation}\label{reprentation-4}
T_g (\dot{f})
=\int_{\partial\Omega} \left\{
G_{s\ell} (n_s D_\ell -n_\ell D_s)f +G_0 f +H \frac{\partial f}{\partial n}\right\} \, d\sigma
\end{equation}
and
\begin{equation}\label{norm-estimate}
\| G_{s\ell}\|_p +\| G\|_p +\| H\|_p \le C\, \| T_g\|
\le C\, \big\{ \| \nabla_t \nabla g\|_p +\| \nabla g\|_p +\| g\|_p\big\}.
\end{equation}
As a consequence, we obtain
$$
\Lambda[g, D_jD_k \Gamma^x]
=\int_{\partial\Omega} \left\{
G_{s\ell} (n_s D_\ell -n_\ell D_s)D_jD_k \Gamma^x
 +G_0 D_jD_k \Gamma^x
 +H \frac{\partial }{\partial n} D_jD_k \Gamma^x \right\} d\sigma
$$
and
\begin{equation}\label{norm-estimate-1}
\|(\Lambda[g, D_jD_k \Gamma^x])^*\|_p
\le C \bigg \{
 \| \nabla_t \nabla g\|_p +\| \nabla g\|_p +\| g\|_p\bigg\}.
\end{equation}

In view of (\ref{Green's-representation-3}), (\ref{estimate of Ijk}) and (\ref{norm-estimate-1}),
we have proved that
$$
\|(\nabla^2 u)^*\|_p \le C\bigg\{ \|\nabla_t\nabla g\|_p +\| \nabla g\|_p +\| g\|_p\bigg\}.
$$
The estimates of $\| (\nabla u)^*\|_p $ and $\|(u)^*\|_p$ are much easier and hence omitted.

Finally we remark that in the proof of the next theorem, it will be shown that
the bilinear estimate (\ref{bilinear-estimate}) implies the existence of solutions
in $(D)_q$. By Theorem \ref{regularity-uniqueness} this gives the uniqueness in $(R)_p$.
The proof of Theorem  \ref{bilinear-imply-regularity} is now complete.
\end{proof}

\begin{thm}\label{bilinear-imply-Dirichlet}
Let $1<q<\infty$ and $\frac{1}{p}+\frac{1}{q}=1$.
Suppose the bilinear estimate (\ref{bilinear-estimate})
holds for any $f,g\in C_0^\infty(\rn{d})$.
Then the $L^q$ Dirichlet problem in $\Omega$ is uniquely solvable.
\end{thm}

\begin{proof}
It follows from the proof of Theorem \ref{bilinear-imply-regularity} that
the bilinear estimate (\ref{bilinear-estimate}) implies the
existence of solutions in $(R)_p$.
By Theorem \ref{Dirichlet-uniqueness} this gives the uniqueness in $(D)_q$.

To establish the existence in $(D)_q$, we will show that for any $f\in C_0^\infty(\rn{d})$,
the unique solution of the $L^2$ regularity problem with boundary data
$D^\alpha f$ satisfies the estimate
\begin{equation}\label{Dirichlet-estimate-1}
\|(\nabla u)^*\|_q
\le C\big\{ \| \nabla f\|_q +\| f\|_q \big\}.
\end{equation}
The existence of solutions with general data in $\mbox{\emph{WA}}^{1,q}(\partial\Omega)$
follows by a standard approximation argument.

It follows from (\ref{Green's-representation}) that
\begin{equation}\label{Green's-representation-4}
\aligned
 D_j u(x)= &
-\int_{\partial\Omega_m} \left\{
\frac{\partial}{\partial n} \Delta D_j \Gamma^x \cdot u
-\Delta D_j \Gamma^x \cdot \frac{\partial u}{\partial n}\right\} d\sigma\\
&\quad\quad\quad
-\int_{\partial\Omega_m}
\left\{ \frac{\partial }{\partial n} D_j \Gamma^x
\cdot \Delta u
-D_j \Gamma^x \cdot \frac{\partial}{\partial n}\Delta u\right\} d\sigma.
\endaligned
\end{equation}
To handle the first term in the right hand side of (\ref{Green's-representation-4}),
we use the identity
$$
\frac{\partial}{\partial n}\Delta D_j \Gamma^x \cdot u
=(n_k D_j -n_j D_k)\big\{ D_k \Delta \Gamma^x \cdot u\big\}
-D_k \Delta \Gamma^x \cdot (n_kD_j -n_j D_k)u
$$
and integration by parts to obtain
\begin{equation}\label{Green's-representation-5}
\aligned
 D_j u(x)= &
\int_{\partial\Omega_m} \left\{
D_k \Delta \Gamma^x \cdot (n_kD_j-n_j D_k)u +D_j \Delta \Gamma^x \cdot \frac{\partial u}{\partial n}
\right\} d\sigma\\
&\quad\quad\quad
-\int_{\partial\Omega_m}
\left\{ \frac{\partial }{\partial n} D_j \Gamma^x
\cdot \Delta u
-D_j \Gamma^x \cdot \frac{\partial}{\partial n}\Delta u\right\} d\sigma.
\endaligned
\end{equation}
Letting $m\to\infty$ in (\ref{Green's-representation-5}) gives
\begin{equation}\label{Green's-representation-6}
 D_j u(x)= I_j(x)
+\Lambda[f, D_j \Gamma^x],
\end{equation}
where
$$
I_j(x)=\int_{\partial\Omega} \left\{
D_k \Delta \Gamma^x \cdot (n_kD_j-n_j D_k)u +D_j \Delta \Gamma^x \cdot \frac{\partial u}{\partial n}
\right\} d\sigma.
$$
It is easy to see that $\|(I_j)^*\|_q \le C\, \| \nabla f\|_q$.

To estimate the nontangential maximal function of the second term in the right hand side of
(\ref{Green's-representation-6}), we observe that by the bilinear estimate (\ref{bilinear-estimate}),
$$
S_f(g) =\Lambda[f, g]=\Lambda[g,f]
$$
defines a bounded linear functional on the space $X^{2,p}(\partial\Omega)$ and its norm is bounded by
$C\, \{ \|\nabla f\|_q +\| f\|_q\}$. Here
$X^{2,p}(\partial\Omega)$ is the completion of
$$
\big\{ \dot{g}=(g_\alpha)_{|\alpha|\le 1}
=(D^\alpha g|_{\partial\Omega})_{|\alpha|\le 1}: \ g\in C_0^\infty(\rn{d})\big\}
$$
under the norm
$$
\| \dot{g}\|_{X^{2,p}(\partial\Omega)}
=\sum_{|\alpha|=1} \|\nabla_t g_\alpha\|_p
+\sum_{|\alpha|\le 1} \| g_\alpha\|_p.
$$
We further note that using the map
$$
\dot{g}=(g_\alpha)_{|\alpha|\le 1}
\to \big( g_\alpha, (n_iD_k-n_k D_i)g_\alpha\big)_{|\alpha|\le 1, 1\le i<k\le d},
$$
the space $X^{2,p}(\partial\Omega)$ may be regarded as a subspace of
$Y=L^p(\partial\Omega)\times \cdots \times L^p(\partial\Omega)$.
By the Hahn-Banach Theorem, $S_f$ extends to a bounded linear functional on $Y$.
It follows that there exist
functions $H_\alpha, H_{\alpha, i, k}$ in $L^q(\partial\Omega)$ such that
$$
S_f (\dot{g})
=\sum_{|\alpha|\le 1} \int_{\partial\Omega} g_\alpha H_\alpha\, d\sigma
+\sum_{\alpha, i, k}
\int_{\partial\Omega}
(n_iD_k-n_kD_i)g_\alpha \cdot H_{\alpha, i, k} \, d\sigma
$$
and
$$
\sum_{|\alpha|\le 1}\| H_\alpha \|_q
+\sum_{\alpha, i, k} \| H_{\alpha, i, k} \|_q
\le C\, \| S_f \|
\le C\bigg\{
\|\nabla f\|_q +\| f\|_q\bigg\}.
$$
Thus
$$
\Lambda [f, D_i \Gamma^x]
=\sum_{|\alpha|\le 1} \int_{\partial\Omega} D^\alpha D_i \Gamma^x \cdot H_\alpha\, d\sigma
+\sum_{\alpha, i, k}
\int_{\partial\Omega}
(n_iD_k-n_kD_i)D^\alpha D_i \Gamma^x \cdot H_{\alpha, i, k} \, d\sigma
$$
and as a consequence, we obtain
\begin{equation}\label{estimate of I}
\| \big(\Lambda[f, D_i \Gamma^x]\big)^*\|_q
\le  C \bigg\{
\sum_{|\alpha|\le 1}\| H_\alpha \|_q
+\sum_{\alpha, i, k} \| H_{\alpha, i, k} \|_q\bigg\}
\le C\bigg\{
\|\nabla f\|_q +\| f\|_q\bigg\}.
\end{equation}
This, together with the estimate for $I_j(x)$, gives the desired estimate (\ref{Dirichlet-estimate-1}).
\end{proof}

\section{Proof of Theorems \ref{Main-Theorem-1}, \ref{self-improving} and \ref{convex-theorem}}

By combining Theorems \ref{Dirichlet-imply-bilinear}, \ref{regularity-bilinear},
\ref{bilinear-imply-regularity} and \ref{bilinear-imply-Dirichlet}, we obtain
Theorem \ref{Main-Theorem-1} and thus Theorem \ref{Main-Theorem}.

It follows from Theorem 1.1 in \cite{shen:necsuff} that for any $q>2$, $(D)_q$ in $\Omega$
is solvable if and only if there exist $C_0>0$ and $r_0>0$ such that
for any $Q\in \partial\Omega$ and $0<r<r_0$,
the weak reverse H\"older condition
\begin{equation}\label{reverse-Holder-1}
\left\{\frac{1}{r^{d-1}}
\int_{I(Q,r)} |(\nabla u)^*|^q \, d\sigma\right\}^{1/q}
\le C_0
\left\{\frac{1}{r^{d-1}}
\int_{I(Q,2r)} |(\nabla u)^*|^2 \, d\sigma\right\}^{1/2},
\end{equation}
holds for any biharmonic function $u$ in $\Omega$ with the properties
that $(\nabla u)^*\in L^2(\partial\Omega)$ and $u=|\nabla u|=0$
on $I(Q,3r)$. Similarly, by Theorem 1.1 in \cite{kilty:higherOrderRegularity},
given any $p>2$, $(R)_p$ in $\Omega$ is solvable if and only if
there exist $C_1>0$ and $r_1>0$ such that
for any $Q\in \partial\Omega$ and $0<r<r_1$,
the weak reverse H\"older condition
\begin{equation}\label{reverse-Holder-2}
\left\{\frac{1}{r^{d-1}}
\int_{I(Q,r)} |(\nabla^2 u)^*|^p \, d\sigma\right\}^{1/p}
\le C_1
\left\{\frac{1}{r^{d-1}}
\int_{I(Q,2r)} |(\nabla^2 u)^*|^2 \, d\sigma\right\}^{1/2},
\end{equation}
holds for any biharmonic function $u$ in $\Omega$ with the properties
that $(\nabla^2 u)^*\in L^2(\partial\Omega)$ and $u=|\nabla u|=0$
on $I(Q,3r)$.
Hence, by Theorem \ref{Main-Theorem}, the solvability of $(D)_q$ for $q<2$
is equivalent to the weak reverse H\"older condition (\ref{reverse-Holder-2})
with $p=\frac{q}{q-1}$.
Theorem \ref{self-improving} now follows from the well known self-improving
property of the weak reverse H\"older inequalities.

Finally we give the proof of Theorem \ref{convex-theorem}.

\begin{proof} (of Theorem \ref{convex-theorem}).
Let $\Omega$ be a convex domain in $\rn{d}$, $d\ge 2$.
 It was proved in \cite{shen:biharmonic} that
$(D)_q$ in $\Omega$ is solvable for $2< q<\infty$. By Theorem \ref{Main-Theorem},
this implies that $(R)_p$ is solvable for $1<p<2$. To establish the solvability of
$(D)_q$ and $(R)_p$ for the remaining ranges, we need to use results in
\cite{kilty:higherOrderRegularity,mayboroda}.

Let $u$ be a solution of the $L^2$ regularity problem in $\Omega$.
Suppose that $u=|\nabla u|=0$ in $B(P,5r)\cap \Omega$ for some
$P\in \partial\Omega$ and $0<r<r_0$.
It follows from Theorem 1.1 in \cite{mayboroda} that
$$
\sup_{B(P,r)\cap\Omega} |\nabla^2 u|
\le \frac{C}{r^2} \left\{\frac{1}{r^d}
\int_{B(P,3r)\cap \Omega}|u|^2\, dx \right\}^{1/2}.
$$
This, together with the assumption $u=|\nabla u|=0$ on $B(P,5r)\cap \partial\Omega$,
 gives that
\begin{equation}
\label{reverseHolder}
\sup_{B(P,r)\cap \Omega} |\nabla^2 u|
\le C \left\{ \frac{1}{r^{d-1}}
\int_{B(P,3r)\cap \partial\Omega}
|(\nabla^2 u)^*|^2\, d\sigma\right\}^{1/2}.
\end{equation}
By Theorem 1.1 in \cite{kilty:higherOrderRegularity}, estimate (\ref{reverseHolder}) implies
the solvability of $(R)_p$ for any $2<p<\infty$.
In view of Theorem \ref{Main-Theorem}, we also
obtain the solvability of $(D)_q$ for $1<q<2$.
As a result, the $L^p$ Dirichlet and regularity problems in $\Omega$
are solvable for any $1<p<\infty$.
Finally we note that the weak maximum principle
(\ref{weak-max})
follows from the solvability of $(R)_p$ and $(D)_q$ for some $p>d-1$
and $q<\frac{d-1}{d-2}$,
by an argument of Pipher and Verchota (see e.g. \cite{pv:maximum}).
\end{proof}

\bibliography{ks32}

\small
\noindent\textsc{Joel Kilty, Department of Mathematics,
Centre College, Danville, KY 40422}\\
\emph{E-mail address}: \texttt{jkilty@ms.uky.edu} \\

\noindent\textsc{Zhongwei Shen, Department of Mathematics,
University of Kentucky, Lexington, KY 40506}\\
\emph{E-mail address}: \texttt{zshen2@email.uky.edu} \\

\noindent \today

\end{document}